\documentclass[psamsfonts]{amsart}
\usepackage{amsmath,amstext,amssymb,amsopn,amsthm}
\usepackage{amsmath,amssymb,amsthm}
\usepackage[mathscr]{eucal}
\usepackage{bbm}
\usepackage{color}
\usepackage{tikz}
\usetikzlibrary{arrows}
\usepackage[unicode,bookmarks,colorlinks]{hyperref}
\hypersetup{
    linkcolor=brickred,
}

\definecolor{mahogany}{cmyk}{0, 0.77, 0.87, 0}
\definecolor{salmon}{cmyk}{0, 0.53, 0.38, 0}
\definecolor{melon}{cmyk}{0, 0.46, 0.50, 0}
\definecolor{yellowgreen}{cmyk}{0.44, 0, 0.74, 0}
\definecolor{brickred}{cmyk}{0, 0.89, 0.94, 0.28}
\definecolor{OliveGreen}{cmyk}{0.64, 0, 0.95, 0.40}
\definecolor{RawSienna}{cmyk}{0, 0.72, 1.0, 0.45}
\definecolor{ZurichRed}{rgb}{1, 0, 0} 

\numberwithin{equation}{section}

\newtheorem{thm}{Theorem}[section]        

\newtheorem{lem}{Lemma}[section]
\newtheorem{prop}{Proposition}[section]
\newtheorem{rmk}{Remark}[section]
\newtheorem{definition}{Definition}[section]
\newtheorem{example}{Example}[section]

\newenvironment{•}{•}{•}
\newcommand{\R}{\mathbb{R}}                  
\newcommand{\Rd}{\R^d}               
\newcommand{\iord}{\int_{\Rd}}              
\newcommand{\intd}[1]{\int_{#1} }    
\newcommand{\set}[1]{ \left\{#1\right\} }
\newcommand{\mysum}[3]{\sum\limits_{#1=#2}^{#3}}          

\newcommand{\wh}[1]{\widehat{#1}}              
\newcommand{\abs}[1]{\left|#1\right|}

\newcommand{\tgo}{t \rightarrow 0+}

\newcommand{\palp}{p_{t}^{(\alpha)}}
\newcommand{\Halp}{\mathbb{H}_{\Omega}^{(\alpha)}(t)}

\newcommand{\ld}{\abs{\Omega}}

\newcommand{\pthesis}[1]{\left(#1\right)}

\newcommand{\bd}{\partial\Omega}

\newcommand{\myH}{\mathbb{H}}

\newcommand{\myP}{\mathcal{P}}
\newcommand{\ck}[1]{p^{(1)}_{#1}}
\newcommand{\dom}{\Omega}
\newcommand{\cvf}{g_{\dom}}
\newcommand{\indf}[1]{\mathbbm{1}_{#1}}
\newcommand{\mylim}[2]{\lim\limits_{#1}{#2}}
\newcommand{\per}{Per(\dom)}
\newcommand{\lm}{\mathcal{L}(\Rd)}
\newcommand{\lo}{\ell_{\dom}}
\newcommand{\lone}[1]{||#1||_{L^{1}(\Rd)}}
\usepackage[margin=1.25in]{geometry}
\begin{document}

\title[ Heat content ]{Heat content estimates over sets  of finite perimeter.}
\author{Luis Acu\~na Valverde}
\address{Department of Mathematics, Universidad de Costa Rica, San Jos\'e, Costa Rica.}
\email{luis.acunavalverde@ucr.ac.cr/ guillemp22@yahoo.com}
\maketitle

\begin{abstract}
This paper studies by means of standard analytic tools   the small time behavior of the heat content over a bounded Lebesgue measurable set of finite perimeter by working with the set covariance function and by imposing conditions on the heat kernels. Applications concerning the heat kernels of rotational invariant $\alpha$-stable processes are given.
\end{abstract}
{\footnotesize {\bf Keywords}: covariance function, heat content, functions of bounded variation, stable processes, sets of finite perimeter.}

\section{introduction}
Let $I$ be a set of indices and $d\geq 2$ an integer. Consider a set of non-negative functions $$\set{p_t^{(\alpha)}(\cdot):\Rd\rightarrow [0,\infty], \alpha\in I, t\geq 0},$$ where each $p_t^{(\alpha)}(\cdot)$ will be called {\it heat kernel}. We shall assume that these heat kernels satisfy the following properties.

\begin{enumerate}
\item[(i)] For each $t>0$, $\palp(x)$ is  radial. That is, $\palp(x)=\palp(\abs{x})\geq 0$, $x\in \Rd$. Furthermore, we assume $p_t^{(\alpha)}(\cdot)\in L^{1}(\Rd)$.
\item[(ii)] Scaling Property: for each integer $d\geq 2$ and $\alpha\in I$, there exist $\beta=\beta(d,\alpha)\in \R$ and $\gamma=\gamma(d,\alpha)>0$  such that
\begin{align}\label{sp}
p_t^{(\alpha)}(x)=t^{\beta}p_{1}^{(\alpha)}(t^{-\gamma}x).
\end{align}
\end{enumerate}

As a consequence of  the aforementioned properties, we obtain
\begin{align}\label{scaling}
\lone{\palp}&=t^{\beta+d\gamma}\,\lone{p_1^{(\alpha)}},\\ \nonumber
\palp(x)&=\palp(\abs{x}e_d),
\end{align}
where $e_d$ stands for  the vector $(0,0,...,0,1)\in \Rd$.

Before continuing, we provide some useful notations. Throughout the paper, $\lm$ will denote   the set of all the Lebesgue measurable subsets of $\Rd$.  For a bounded  set $\dom \in \lm$  with non-empty boundary $\bd$, we set 
\begin{align*}
\ld&=\mbox{volume of } \dom,\\
\mathcal{H}^{d-1}(\bd)&= \mbox{$(d-1)$--Hausdorff measure of the boundary of}\,\, \dom.
\end{align*}

Henceforth, $B_r(x)$ will stand for  the ball centered at  $x\in\Rd$ with radius $r$ and for simplicity $B$ will represent the unit ball centered at zero. Also $S^{d-1}$ will denote the boundary of the unit ball $B$. Moreover, the volume  and   surface area of the unit ball in $\Rd$ will be denoted by
$w_{d}$ and $A_d$, respectively. That is,
\begin{align}\label{vaunitball}
w_d&=\frac{\,\pi^{\frac{d}{2}}}{\Gamma\pthesis{1+\frac{d}{2}}}, \\ \nonumber
A_d&=dw_d.
\end{align}

In addition, if  $g:\dom\subseteq\Rd \rightarrow \R$ is a Lipschitz function, we denote 
\begin{align*}
Lip(g)=\sup\set{  \frac{\abs{g(y)-g(x)}}{\abs{y-x}}:x,y\in \dom, x\neq y}.
\end{align*}

Let  $\dom \in \lm$ be a bounded set. The  purpose of the paper is to investigate the behavior as  $\tgo$ of the following function 
\begin{align}\label{defhc}
\Halp=\intd{\dom}dx\intd{\dom}dy\,\palp(x-y),
\end{align} 
which will be called {\it the heat content} of $\dom$ in $\Rd$ by imposing conditions over the heat kernel $\palp(\cdot)$ and the underlying set $\dom$. We remark that $\Halp$ is finite for all $t>0$ due to the assumption $\palp(\cdot)\in L^{1}(\Rd)$ and the inequality
$$0\leq \Halp\leq \intd{\dom}dx\intd{\Rd}dy\,\palp(x-y)= \ld\,\lone{\palp}.$$

The function $\Halp$ turns out to provide information about the geometry  of the  set $\dom$ as long as regularity conditions over $\dom$ are assumed. For instance, in \cite[Theorem 2.4]{Mir}  is proved  by taking $I=\set {2}$  and  considering the Gaussian kernel
\begin{align}\label{gk}
p_t^{(2)}(x)=(4\pi t)^{-\frac{d}{2}}\exp\pthesis{-\frac{\abs{x}^2}{4t}}
\end{align}
that
\begin{equation*}
\lim\limits_{\tgo}\frac{\ld-\myH^{(2)}_{\Omega}(t)}{\sqrt{t}}=\frac{1}{\sqrt{\pi}}\mathcal{H}^{d-1}(\bd),
\end{equation*}
for $\Omega$ a bounded domain with boundary being $C^2$.
In \cite{vanden3}, M. van den Berg  called $\myH^{(2)}_{\dom}(t)$  the heat content of $\Omega$ in $\Rd$ and therefore  following the terminology introduced by M. van den Berg, we have also called $\Halp$  the heat content of $\Omega$ in $\Rd$.  We refer the interested reader to the papers \cite{van6, van7, van8} for recent results concerning bounds and asymptotic behaviors of the heat content corresponding to the case  $I=\set{2}$ and the Gaussian kernel over open sets, polygonal  domains and its extensions when dealing with compact manifolds.

In order to investigate the small time behavior of $\Halp$, we need to introduce the notion of finite perimeter. We say that a bounded set $\dom \in \lm$ has finite perimeter if 
\begin{align}\label{defper}
0\leq  \sup\set{\intd{\dom}dx\,{\bf div} \varphi(x): \varphi\in C^{1}_{c}(\Rd,\Rd), ||\varphi||_{\infty}\leq 1}<\infty,
 \end{align}     
 and we denote the last quantity by $Per(\Omega)$.

Our main result is the following.
\begin{thm}\label{mthm}  Let $d\geq 2$ be an integer. Consider $\dom \in \lm$ a bounded set with $\per<\infty$ and  let $w_{d-1}$  and $A_d$ be the constants defined in \eqref{vaunitball}. 
\begin{enumerate}
\item[$ (i)$]Let $I_0=\set{\alpha\in I: |\cdot|\,p_1^{(\alpha)}(\cdot)\in L^{1}(\Rd)}.$ For each $\alpha \in I_0$, we have for all $t>0$ that 
\begin{align}\label{i1}
\lone{p_1^{(\alpha)}}\ld- t^{-(\beta+d\gamma)}\Halp\leq t^{\gamma} w_{d-1}\per\int_{0}^{\infty}dr\,r^d\,p_1^{(\alpha)}(r\,e_d).
\end{align}
Furthermore,
\begin{align}\label{lim1}
\mylim{\tgo}{\frac{\lone{p_1^{(\alpha)}}\ld- t^{-(\beta+d\gamma)}\Halp}{t^{\gamma}}}=w_{d-1}\per\int_{0}^{\infty}dr\,r^d\,p_1^{(\alpha)}(r\,e_d).
\end{align}

\item[$(ii)$] 
Let $$I_1=\set{\alpha \in I:
p^{(\alpha)}_{t}(x)=
\frac{\kappa \,t^{\beta}}
{\pthesis{1+\abs{t^{-\gamma}x}^{n}}^{m}}, \,\,d-nm=-1,\,\,n,m, \kappa>0},$$ and  denote the diameter of $\dom$ defined as
$\sup\set{\abs{x-y}:x,y\in \dom}$ by $\lo$.  Then, for all $t>0$ satisfying $t^{\gamma}<\lo$, we have
\begin{align*}
\lone{p_1^{(\alpha)}}\ld- t^{-(\beta+d\gamma)}\Halp\leq t^{\gamma}\pthesis{\lambda(\dom)+\kappa\,w_{d-1}\per\,\gamma\,
\ln\pthesis{\frac{1}{t}}},
\end{align*}
where 
$$\lambda(\dom)=\ld\, \lo^{-1}\,A_d\,\kappa+ \kappa\,w_{d-1}\per\,\pthesis{\ln(\lo)+\int_{0}^{1}\frac{dr\,r^d}{(1+r^n)^m}}.$$
In particular, we arrive at
\begin{align*}
\varlimsup_{\tgo}\frac{\lone{p_1^{(\alpha)}}\ld- t^{-(\beta+d\gamma)}\Halp}{t^{\gamma}\ln\pthesis{\frac{1}{t}}}\leq \kappa\,\per\,w_{d-1}\gamma.
\end{align*}

\item[$(iii)$]
Let $I_{2}\subset I$ be such that for each $\alpha \in I_2$, there are functions $\phi_{\alpha}:[0,1]\rightarrow [0, \infty)$,
$J_{\alpha}:\R^{d}\rightarrow (0, \infty]$  and $\Lambda_{\alpha}:\R^{d}\rightarrow (0, \infty]$ satisfying
\begin{align}\label{bc}
\frac{p_t^{(\alpha)}(x)}{\phi_{\alpha}(t)}\leq \Lambda_{\alpha}(x), \,\,x\neq 0, \,0<t\leq 1,
\end{align}
\begin{align}\label{lc}
\lim\limits_{\tgo}\frac{p_t^{(\alpha)}(x)}{\phi_{\alpha}(t)}=J_{\alpha}(x), \,x\neq 0,
\end{align}
and
\begin{align}\label{fi1}
\int_{\Omega}dx \int_{\Omega^c}dy\,J_{\alpha}(x-y)<\infty,\,\,\,\,
\int_{\Omega}dx\int_{\Omega^c}dy\,\Lambda_{\alpha}(x-y)<\infty.
\end{align}
\end{enumerate}
Then,
\begin{align*}
\mylim{\tgo}{\frac{t^{d\gamma+\beta}\lone{p_1^{(\alpha)}}\ld- \Halp}{\phi_{\alpha}(t)}}=\int_{\Omega}dx\int_{\Omega^c}dy\, J_{\alpha}(x-y).
\end{align*}
\end{thm}
We remark that the different small time behaviors provided in the foregoing theorem implicitly contains the fact that $I_0,I_1$ and $I_2$ are assumed to be disjoint subsets of $I$ and at least one of them is not empty.

The key step to proving Theorem \ref{mthm} consists on expressing the heat content $\Halp$ in terms of the set covariance function of $\dom$ defined in \eqref{gdef} below which as we shall see in the next section is linked to the perimeter of $\dom$.

A classical example of heat kernels satisfying all the above assumptions
are the transition probabilities corresponding to the rotational invariant $\alpha$--stable process whose main properties will be described in
\S \ref{sec:stable} below. There is an increasing interest in investigating 
small and large time behavior of functions related to the transition densities of a stable process and we refer the reader to  \cite{Acu, Acu0,Acu3, BanKul, BanKulSiu, BanYil, van5} for recent developments concerning the heat trace, heat content and spectral heat content for 
bounded domains with smooth boundary and Schr\"{o}dinger operators on $\R^d$. 
 
The paper is organized as follows.  In \S \ref{sec:prelim}, we introduce the geometric objects associated with sets $\dom$ of finite perimeter. Namely, set covariance and bounded variation functions. In \S \ref{sec:proof}, we provide the proof of Theorem \ref{mthm}. In \S \ref{sec:stable}, an application of Theorem \ref{mthm} is given when working with the heat kernels of a rotational invariant $\alpha$-stable process. Finally, in \S \ref{sec:poisson}, the heat content over the unit ball related to the Poisson kernel (see \eqref{Cauchyk} below) is investigated.

\section{preliminaries: functions of bounded variation, perimeter and covariance function.}\label{sec:prelim}
In this section, we introduce a couple of  geometric objects associated with the set $\dom$ under consideration  which will play an important role in the proof of
Theorem \ref{mthm}. The interested reader may consult \cite{Evans}, \cite{Lau}, \cite{Mor}  and \cite{Mir}  for further details on the matter and for the proofs of the many results to be given in this section.

\begin{definition}
Let $G\subseteq \Rd$ be an  open set and $f:G\rightarrow \R$, $f\in L^{1}(G)$. The total variation of $f$ in $G$ is defined by
$$V(f,G)=\sup\set{\intd{G}dx\,f(x) {\bf div} \varphi(x): \varphi\in C^{1}_{c}(G,\Rd), ||\varphi||_{\infty}\leq 1}.$$
We set $BV(G)=\set{f\in L^{1}(G):V(f,G)<\infty}$ to denote the set of functions of bounded variation.

The directional derivative of $f$ in $G$ in the direction $u\in S^{d-1}$ is
$$V_{u}(f,G)=\sup\set{\intd{G}dx\,f(x) \langle\nabla \varphi(x),u\rangle: \varphi\in C^{1}_{c}(G,\Rd), ||\varphi||_{\infty}\leq1}.$$

If $\dom \in \lm$, we call $V(\indf{\dom}, \Rd)$ the perimeter of $\dom$ and we denote this  quantity by $Per(\dom)$. In addition, $V_{u}(\dom)$ will denote for simplicity the quantity $V_u(\indf{\dom}, \Rd)$.
\end{definition}

In order to gain an insight into  functions of bounded variation, we proceed to provide some  classical examples.
\begin{example}
Let $G\subseteq\Rd$ be an open set and consider 
$$W^{1,1}(G)=\set{f\in L^{1}(G): 
\frac{\partial f}{\partial x_ j}\in L^{1}(G)}.$$
Then $W^{1,1}(G)\subseteq BV(G)$.  To prove  this, it suffices to see that by integration by parts formula, we have for any $\varphi\in C^{1}_{c}(G,\Rd)$ with $ ||\varphi||_{\infty}\leq 1$ that
$$\intd{G}dxf(x){\bf div} \varphi(x)=-\intd{G}dx\langle \nabla f(x), \varphi(x) \rangle\leq \intd{G} dx\abs{\nabla f(x)}<\infty.$$
Thus, 
\begin{align}\label{eq1}
V(f,G)\leq \intd{G} dx\abs{\nabla f(x)}.
\end{align}
It is worth mentioning  that equality  indeed happens in \eqref{eq1} and we refer the reader to \cite{Mor} for the proof.
\end{example}

The following example tells us that the perimeter of a bounded set and  the $(d-1)$-Hausdorff measure  of the boundary are related provided that the boundary is smooth enough. 

\begin{example}\label{Smoothbd}
Let $\dom \in \lm$ be a bounded open set  with $C^2$ boundary  $\bd$ and $\mathcal{H}^{d-1}(\bd)<\infty$. Then, we claim that
$\indf{\dom}\in BV(\dom)$ and 
$$\per=V(\indf{\dom},\Rd)=\mathcal{H}^{d-1}(\bd).$$
To see this, we apply the divergence theorem to any $\varphi\in C^{1}_{c}(\Rd,\Rd)$ with $ ||\varphi||_{\infty}\leq 1$ to obtain
\begin{align}\label{diveq}
\intd{\dom}dx\,{\bf div} \varphi(x)=\intd{\bd}\langle \varphi(x),n(x)\rangle \mathcal{H}^{d-1}(dx),
\end{align}
where $n(x)$ is the unit normal along $\bd$.  The last identity implies $\per\leq \mathcal{H}^{d-1}(\bd)$ since 
$\langle \varphi(x),n(x)\rangle \leq \abs{\varphi(x)}\abs{n(x)}\leq 1.$

On the other hand, the facts that $\dom$ is bounded  and $\bd$ is $C^2$ allow us   to construct  a compact set $K$  such that $\dom \subset K$ and a vector field $\varphi_{\dom}\in C^{1}_{c}(\Rd,\Rd)$ with $ ||\varphi_{\dom}||_{\infty}\leq 1$ satisfying $\varphi_{\dom}(x)=n(x)$ for $x\in \bd$ 
and $\varphi_{\dom}(x)=0$ for $x\in K^c$. Hence, an application of
\eqref{diveq} yields
$$\intd{\dom}dx\,{\bf div}\varphi_{\dom}(x)=\intd{\bd}\langle n(x),n(x)\rangle \mathcal{H}^{d-1}(dx)= \mathcal{H}^{d-1}(\bd),$$
which in turn finishes the proof of our claim.
\end{example}

The following geometric object will be essential in order to investigate the small time behavior of the heat content $\Halp$.
\begin{definition}[\bf Set covariance function]
Let $\dom\in\lm$ have finite Lebesgue measure. The covariance function of $\dom$ is denoted by $\cvf$ and defined for each $y\in \Rd$ by
\begin{align}\label{gdef}
\cvf(y)=\abs{\dom \cap \pthesis{\dom+y}}=\iord dx\,\indf{\dom}(x)\,\indf{\dom}(x-y).
\end{align}
\end{definition}
We now proceed  to mention some analytic properties associated with the set covariance function $\cvf$.
\begin{prop}\label{gprop}
Let $\dom \subset \Rd$ be a Lebesgue measurable set with $\ld<\infty$ and $\cvf$ its corresponding covariance function. Then,
\begin{enumerate}
\item[$a)$] For all $y\in \Rd$, $0\leq \cvf(y)\leq \cvf(0)=\ld$.
\item[$b)$] For all $y\in \Rd$, $ \cvf(y)= \cvf(-y)$.
\item[$c)$] $\iord dy\,\cvf(y)=\ld^2$.
\item[$d)$] $\cvf$ is compactly supported. In fact,  for all $y \in \Rd$ with $\abs{y}\geq \lo=\sup\set{\abs{x-y}:x,y\in \dom},$ we have
$\cvf(y)=0$.
\item[$e)$] $\cvf$ is uniformly continuous over $\Rd$ and $\mylim{\abs{y}\rightarrow \infty}{\cvf(y)}=0$.
\end{enumerate}
\end{prop}

The following propositions reveal the link among functions of bounded variation, directional variation and sets of finite perimeter.  The proof of these results can be found in \cite{Galerne}.

\begin{prop}Let $G$ be an open subset of $\Rd$ and consider $f\in L^{1}(\Rd)$. Then, $V(f,G)$ is finite if and only if  the directional variation $V_{u}(f,G)$ is finite for every direction $u\in S^{d-1}$ and
\begin{align}\label{relation}
V(f,G)=\frac{1}{2w_{d-1}}\intd{S^{d-1}}\mathcal{H}^{d-1}(du)\,V_{u}(f,G).
\end{align}
In particular, for any $\dom\in \lm$  with finite perimeter, we have
\begin{align}\label{perident}
\per=\frac{1}{2w_{d-1}}\intd{S^{d-1}}\mathcal{H}^{d-1}(du)\,V_{u}(\dom).
\end{align}
\end{prop}

\begin{prop}\label{derlip} Let $\dom \in \lm$ be such that $\ld$ is finite and consider $\cvf$ its corresponding covariance function and $u\in S^{d-1}$.  The following assertions are equivalent.
\begin{enumerate}
\item[$(i)$] $V_{u}(\dom)$ is finite.
\item[$(ii)$] $
\mylim{r\rightarrow 0}{\frac{\cvf(0)-\cvf(ru)}{\abs{r}}}
$ exists and is finite.
\item[$(iii)$] The real valued function $\cvf^{u}(r)=\cvf(ru)$ is Lipschitz.
Moreover,
\begin{align*}
Lip(\cvf^{u})=\mylim{r\rightarrow 0}{\frac{\cvf(0)-\cvf(ru)}{\abs{r}}}=\frac{V_u(\Omega)}{2}.
\end{align*}
\end{enumerate}

\end{prop}

\begin{prop}\label{fpg}
Let $\dom \in \lm$ be a bounded set with $\per<\infty$ and consider $\cvf$ its corresponding covariance function. Then,
\begin{enumerate}
\item[i)]$\cvf$ is Lipschitz with
$$Lip(\cvf)=\frac{1}{2}\sup_{u\in S^{d-1}}{V_u(\dom) }\leq \frac{1}{2}Per(\dom).$$
\item[ii)] For each $u\in S^{d-1}$ and $r>0$, 
$$\pthesis{g_{\dom}^{u}}'(0+)=\mylim{r\rightarrow 0+}{\frac{\cvf(ru)-\cvf(0)}{r}}$$ exists and is finite. Moreover,
\begin{align*}
Per(\dom)=-\frac{1}{w_{d-1}}\int_{S^{d-1}}\mathcal{H}^{d-1}(du)\,\pthesis{g_{\dom}^{u}}'(0+).
\end{align*}
\end{enumerate}
\end{prop}

\section{proof of theorem \ref{mthm} }\label{sec:proof}
We begin this section by  rewriting $\Halp$ in terms of the set covariance function $\cvf$.  By appealing to Fubini's Theorem and performing a simple change of variable, we have  based on  \eqref{defhc} and \eqref{gdef} that
\begin{align*}
\Halp=\iord dx\iord dy \,\palp(x-y)\indf{\dom}(y)\indf{\dom}(x)=\iord dz\,
\palp(z)\,\cvf(z).
\end{align*}
By using the scaling property \eqref{sp} and the change of variable $w=t^{-\gamma}z$, we arrive at
\begin{align}\label{id1}
 t^{-(d\gamma+\beta)}\,\Halp&=\iord dw\,
p_1^{(\alpha)}(w)\cvf\pthesis{t^{\gamma}w}\\ \nonumber
&=
\cvf\pthesis{0}\lone{p_1^{(\alpha)}}+\iord dw\,
p_1^{(\alpha)}(w)\pthesis{\cvf\pthesis{t^{\gamma}w}-\cvf(0)}.
\end{align}
Hence, we have shown that
\begin{align}\label{id3}
\iord dw\,
p_1^{(\alpha)}(w)\pthesis{\cvf(0)-\cvf\pthesis{t^{\gamma}w}}=
\cvf\pthesis{0}\lone{p_1^{(\alpha)}}-t^{-(d\gamma+\beta)}\,\Halp,
\end{align}
where $\cvf(0)=\ld$ by Proposition \ref{gprop}.
Next, with the aid of \eqref{id3}, we start  the proof of Theorem \ref{mthm}.

\bigskip
{\bf Proof of part $(i)$ of of Theorem \ref{mthm}:}
Let us define
\begin{align}\label{Fdef}
F_{\dom}^{(\alpha)}(t)=\iord dw\,
p_1^{(\alpha)}(w)\pthesis{\cvf(0)-\cvf\pthesis{t^{\gamma}w}}.
\end{align}

Now, polar coordinates  and the fact that $\palp(x)=\palp(\abs{x}e_d)$ allow us to express $F_{\dom}^{(\alpha)}(t)$ as follows.
\begin{align}\label{Fident}
F_{\dom}^{(\alpha)}(t)&=\int_{0}^{\infty}dr\,r^{d-1}\,p_1^{(\alpha)}(r\,e_d)\intd{S^{d-1}}\mathcal{H}^{d-1}(du)\pthesis{\cvf\pthesis{0}-\cvf\pthesis{t^{\gamma}r\,u}}\nonumber\\
&=t^{\gamma}
\int_{0}^{\infty}dr\,r^{d}\,p_1^{(\alpha)}(r\,e_d)M_{\Omega}(t,r),
\end{align}
where  
\begin{align}\label{Mdef}
M_{\Omega}(t,r)=\intd{S^{d-1}}\mathcal{H}^{d-1}(du)\pthesis{\frac{\cvf\pthesis{0}-\cvf\pthesis{t^{\gamma}r\,u}}{t^{\gamma}\,r}}.
\end{align}

We first proceed to prove $\eqref{lim1}$.  Notice that by Proposition \ref{fpg}, we have
\begin{align}\label{ine1}
\abs{\frac{\cvf\pthesis{t^{\gamma}r\,u}-\cvf\pthesis{0}}{t^{\gamma}\,r}}\leq \frac{1}{2}\per\in L^{1}(S^{d-1}).
\end{align}
 Therefore,  it follows from the last inequality, part $(ii)$ of Proposition  \ref{fpg} and the Lebesgue Dominated convergence Theorem that
\begin{align}\label{limM}
\mylim{\tgo}{M_{\Omega}(t,r)}=w_{d-1}\per.
\end{align}

On the other hand, it follows from \eqref{Mdef} and \eqref{ine1} that
\begin{align}\label{ineq2}
r^{d}\,p_1^{(\alpha)}(r\,e_d)M_{\Omega}(t,r)\leq \frac{1}{2}A_d \per\, r^{d}\,p_1^{(\alpha)}(r\,e_d).
\end{align}
Notice that $r^{d}\,p_1^{(\alpha)}(r\,e_d)\in L^{1}((0,\infty))$ because the assumption 
$|\cdot|p_1^{(\alpha)}(\cdot)\in L^{1}(\Rd)$ and polar coordinates  imply that
\begin{align}\label{polarcor}
A_d\int_{0}^{\infty}dr \,r^dp_1^{(\alpha)}(r e_d)=\int_{\Rd}dw |w|p_1^{(\alpha)}(w)<\infty.
\end{align}
Hence,  by combining the Lebesgue Dominated convergence Theorem together with the limit \eqref{limM} and inequality \eqref{ineq2}, we conclude by appealing to the identities \eqref{id3} and \eqref{Fdef} that
\begin{align*}
\mylim{\tgo}{\frac{
\cvf\pthesis{0}\lone{p_1^{(\alpha)}}-t^{-(d\gamma+\beta)}\,\Halp}{t^{\gamma}}}&=\mylim{\tgo}
{\frac{F_{\dom}^{(\alpha)}(t)}{t^{\gamma}}}\\
&=w_{d-1}\per\int_{0}^{\infty}dr\,r^d\,p_1^{(\alpha)}(r\,e_d).
\end{align*}

Regarding the inequality $\eqref{i1}$,  by appealing to the definition of $M_{\dom}(t,r)$ provided in \eqref{Mdef} and part $(iii)$ of Proposition \ref{derlip}, we derive from identity \eqref{Fident} that
\begin{align*}
F_{\dom}^{(\alpha)}(t)\leq t^{\gamma}
\int_{0}^{\infty}dr\,r^{d}\,p_1^{(\alpha)}(r\,e_d)\pthesis{\frac{1}{2}\intd{S^{d-1}}\mathcal{H}^{d-1}(du)V_{u}(\dom)}.
\end{align*}
Thus, it follows from \eqref{perident} that
\begin{align*}
F_{\dom}^{(\alpha)}(t)\leq t^{\gamma}\,w_{d-1}\,\per
\int_{0}^{\infty}dr\,r^{d}\,p_1^{(\alpha)}(r\,e_d).
\end{align*}
Finally, observe that \eqref{id3} and \eqref{Fdef} lead to the desired inequality \eqref{i1} and this finishes the proof of part $(i)$ of Theorem \ref{mthm}.\qed

\

The proof of part $(ii)$ cannot follow the same outline of part $(i)$ because for every $d,n,m$ satisfying $d-nm=-1$ and $n,m>0$, we have
\begin{align*}
\int_{1}^{\infty}dr\frac{r^d}{(1+r^{n})^m}\geq 2^{-m}\int_{1}^{\infty}dr \,\frac{r^d}{r^{nm}}=2^{-m} \int_{1}^{\infty}dr \,r^{-1}=\infty,
\end{align*}
where we have used that $1+r^n\leq 2r^{n}$ for all $r\geq1$. Thus, the divergence of the last integral in turn implies by polar coordinates that $\abs{\cdot}\palp(\cdot)\notin L^{1}(\Rd)$(see \eqref{polarcor}) for every heat kernel $\palp(\cdot )$ satisfying assumptions in part $(ii)$ of Theorem \ref{mthm}.
\bigskip

{\bf Proof of part (ii) of Theorem \ref{mthm}:} We recall that along the proof, we assume that  the heat kernels $\palp(x)$ are explicitly given by
\begin{align}\label{expker}
p^{(\alpha)}_{t}(x)=
\frac{\kappa \,t^{\beta}}
{\pthesis{1+\abs{t^{-\gamma}x}^{n}}^{m}},
\end{align}
with $d-nm=-1$ and  $n,m, \kappa>0$.
By part $d)$ in Proposition  \ref{gprop}, we know that $\cvf(y)=0$ if $|y|>\lo$ with $\lo$ being the diameter of $\dom$. Therefore, we arrive at the following decomposition of the heat content.
\begin{align}
t^{-(d\gamma+\beta)}\,\Halp&=\iord dw\,
p_1^{(\alpha)}(w)\cvf\pthesis{t^{\gamma}w}\\ \nonumber
&=\cvf(0)\lone{p_1^{(\alpha)}}-F_1(t)-t^{\gamma}F_2(t),
\end{align}
where
$$F_1(t)=\cvf(0)\intd{\abs{w}\geq\lo t^{-\gamma}}dw\,
p_1^{(\alpha)}(w)$$
and
\begin{align*}
F_2(t)&=\intd{\abs{w}<\lo t^{-\gamma}}dw\,p_1^{(\alpha)}(w)\pthesis{\frac{\cvf(0)-\cvf(t^{\gamma}w)}{t^{\gamma}}}\\ \nonumber
&=\int_{S^{d-1}}\mathcal{H}^{d-1}(du)\int_{0}^{\lo t^{-\gamma}}dr\, r^d \,p_1^{(\alpha)}(r\,e_d)\pthesis{\frac{\cvf\pthesis{0}-\cvf\pthesis{t^{\gamma}r\,u}}{t^{\gamma}\,r}}.
\end{align*}
Now, due to the fact that $1+r^n>r^{n}$ and $d-nm=-1$, we obtain by appealing to polar coordinates and the explicit form of the heat kernels \eqref{expker} that
\begin{align}\label{ineqint}
F_1(t)=\cvf(0)\,A_d\,\kappa\int_{\lo\,t^{-\gamma}}^{\infty}\frac{dr\,r^{d-1}}{(1+r^n)^m}\leq A_d\,\kappa\,\int_{ \lo t^{-\gamma}}^{\infty}dr\,r^{-2}=|\Omega|\,A_d \, \kappa\,\lo^{-1} \,t^{\gamma}.  
\end{align}

On the other hand, by Proposition \ref{derlip}, we obtain that 
\begin{align*}
0\leq\frac{\cvf\pthesis{0}-\cvf\pthesis{t^{\gamma}r\,u}}{t^{\gamma}\,r}\leq \frac{1}{2}V_u(\dom).
\end{align*}
Thus, we deduce by \eqref{perident} and \eqref{expker} that
\begin{align}\label{F2ineq}
F_2(t)\leq \,\per\,w_{d-1}\,\kappa\int_{0}^{\lo\,t^{-\gamma}}\frac{dr\,r^d}{(1+r^n)^m}.
\end{align} 

Now, we observe   that
\begin{align*}
\int_{0}^{\lo t^{-\gamma}} \frac{dr\,r^d}{(1+r^n)^m} = \int_{0}^{1} \frac{dr\,r^d}{(1+r^n)^m}
+\int_{1}^{\lo t^{-\gamma}} \frac{dr\,r^d}{(1+r^n)^m},
\end{align*}
as long as  $t^{\gamma}<\lo$. Therefore,  we arrive  by using once more that $d-nm=-1$ and $(1+r^n)^{m}\geq r^{nm}$ at
\begin{align*}
\int_{1}^{\lo t^{-\gamma}} \frac{dr\,r^d}{(1+r^n)^m}\leq
\int_{1}^{\lo t^{-\gamma}} dr\,r^{-1}=\ln (\lo)+\gamma\,\ln\pthesis{\frac{1}{t}}.
\end{align*}
Hence, by \eqref{F2ineq}, we have proved that
\begin{align}\label{ineF22}
F_2(t)\leq \kappa\, \per \,w_{d-1}\pthesis{\int_{0}^{1} \frac{dr\,r^d}{(1+r^n)^m}+\ln(\lo)+\gamma \ln \pthesis{\frac{1}{t}}}.
\end{align}
Therefore, by combining  the previous inequalities \eqref{ineqint} and \eqref{ineF22} together with the identity
$$\cvf(0)\,\lone{p_1^{(\alpha)}}- t^{-(\beta+d\gamma)}\Halp=F_1(t)+ t^{\gamma}F_2(t),$$
we arrive at  the desired result and this finishes  the proof  pf part $(ii)$ of Theorem \ref{mthm}. \qed
\\

{\bf Proof of part (iii) of Theorem \ref{mthm}:}
By using that $\indf{\dom}(\cdot)=1-\indf{\dom^c}(\cdot)$, we have
\begin{align}
\cvf(y)=\int_{\Rd}dx\, \indf{\dom}(x)\,\indf{\dom}(x+y)=\ld-h_{\dom}(y),
\end{align} 
where
$$h_{\dom}(y)=\int_{\Rd}dx \,\indf{\dom}(x)\,\indf{\dom^c}(x+y).$$

This implies 
\begin{align*}
\Halp=\iord dz\,
\palp(z)\cvf(z)=\ld \lone{p_1^{(\alpha)}}t^{\beta+d\gamma}-\intd{\Rd}dz\, \palp(z)\,h_{\dom}(z).
\end{align*}

It is easy to show by Fubini's Theorem    that 
\begin{align*}
\intd{\Rd}dz\, \palp(z)\,h_{\dom}(z)=\int_{\dom}dx\int_{\dom^c}dy\,\palp(x-y).
\end{align*}
Thus, we arrive at
\begin{align}\label{id4}
\mylim{\tgo}{\frac{\ld \lone{p_1^{(\alpha)}}t^{\beta+d\gamma}- \Halp}{\phi_{\alpha}(t)}}=\mylim{\tgo}{\int_{\dom}dx\int_{\dom^c}dy
\,\frac{\palp(x-y)}{\phi_{\alpha}(t)}}.
\end{align}
The conditions \eqref{bc}, \eqref{lc} and \eqref{fi1} are given to apply the Lebesgue dominated convergence Theorem to  the right hand side of \eqref{id4} which in turn completes the proof of Theorem \ref{mthm}.

\section{applications to rotational invariant $\alpha$-stable processes, $0<\alpha<2.$}\label{sec:stable}
The heat kernels of  rotational invariant $\alpha$-stable processes  are  probability densities $$\set{\palp(\cdot):\Rd\rightarrow [0,\infty]:t\geq 0,\alpha\in I=(0,2)},$$  which are completely determined by their Fourier transform and they satisfy the following properties.
\begin{enumerate}
\item[(i)] For all $\alpha \in (0,2)$, $t>0$ and $x\in \Rd$, we have that $\wh{\palp}(x)=e^{-t\abs{x}^{\alpha}}$. As a result of this Fourier transform,  we see that $\palp(x)$ is radial and $\lone{\palp}=1$.
\item[(ii)] Scaling Property: for each integer $d\geq 2$ and $\alpha\in (0,2)$,
\begin{align}\label{spstable}
p_t^{(\alpha)}(x)=t^{-\frac{d}{\alpha}}p_{1}^{(\alpha)}(t^{-\frac{1}{\alpha}}x).
\end{align}
\end{enumerate}
That is, they satisfy the scaling property $\eqref{sp}$ with $\beta =-\frac{d}{\alpha}$ and $\gamma=\frac{1}{\alpha}$, where for these values we obtain $\beta+d\gamma=0$.

The transition densities $\palp(x-y)$ are known to have  an explicit expression only for  $\alpha=1$.  In fact, for $\alpha=1$, the function
$p^{(1)}_t(x-y)$ is called the Poisson  heat kernel and  it is given
by
\begin{align}\label{Cauchyk}
p^{(1)}_t(x-y)=\frac{k_d\,t}{\pthesis{t^2+\abs{x-y}^2}^{(d+1)/2}},
\end{align}  
where
\begin{align}\label{kappadef}
k_d=\frac{\Gamma\pthesis{\frac{d+1}{2}}}{\pi^{\frac{d+1}{2}}}.
\end{align}

However, for the purposes of this paper, we only need to make use of  the following two facts about $\palp(x-y)$ for all $\alpha \in (0,2)$. First, there exists $c_{\alpha,d}>0$ such that
\begin{align}\label{tcomp}
c_{\alpha,d}^{-1}\min\set{t^{-d/\alpha},\frac{t}{\abs{x-y}^{d+\alpha}}}\leq
\palp(x-y)\leq c_{\alpha,d}\min\set{t^{-d/\alpha},\frac{t}{\abs{x-y}^{d+\alpha}}},
\end{align}
for all $x,y\in \R^d$ and $t>0$ (see \cite{Chen1}).  Secondly, according to \cite[Theorem 2.1]{Blum}, we have
\begin{align}\label{tlim}
\lim_{\tgo}\frac{\palp(x-y)}{t}=\frac{C_{\alpha,d}}{\abs{x-y}^{d+\alpha}},
\end{align}
for all $x\neq y$, where 
\begin{equation}\label{Adef}
C_{\alpha,d}=\alpha \, 2^{\alpha-1}\, \pi^{-1-\frac{d}{2}}\,\sin\pthesis{\frac{\pi\alpha}{2}}\,\Gamma\pthesis{\frac{d+\alpha}{2}}\,\Gamma\pthesis{\frac{\alpha}{2}}.
\end{equation}

One interesting aspect about the heat kernel $\palp(x-y)$ is that it can be written in terms of the Gaussian kernel \eqref{gk}. Namely, for each $\alpha \in (0,2)$, there exist  probability functions denoted by $\set{\eta^{(\alpha/2)}_t}_{t>0}$ satisfying
\begin{align}\label{dbysub}
\palp(x-y)=\int_0^{\infty} ds\,p^{(2)}_s(x-y) \,\eta_{t}^{(\alpha/2)}(s).
\end{align}
Furthermore, in \cite{Acu0} is proved by means of probabilistic techniques that $s^{\lambda}\,\eta^{(\alpha/2)}_1(s)\in \L^{1}((0,\infty))$ if and only if $-\infty <\lambda< \frac{\alpha}{2}$ and
\begin{equation}\label{gen.expt.S1}
\int_{0}^{\infty}\,ds\,s^{\lambda}\,\eta^{(\alpha/2)}_1(s)=\frac{\Gamma(1-\frac{2\lambda}{\alpha})}{\Gamma(1-\lambda)}.
\end{equation}

The aforementioned estimates gives the following result.

\begin{lem}\label{finiteexp}
$\abs{\cdot}p_1^{(\alpha)}(\cdot)\in L^{1}(\Rd)$ if and only if $r^{d}p_1^{(\alpha)}(r\,e_d)\in L^{1}((0,\infty))$ if and only if $\alpha\in (1,2)$. Moreover,
\begin{align}
\int_{0}^{\infty}dr\,r^{d}p_1^{(\alpha)}(r)=\Gamma\pthesis{\frac{d+1}{2}}\pi^{-\frac{d+1}{2}}\Gamma\pthesis{1-\frac{1}{\alpha}}.
\end{align}
\end{lem} 
\begin{proof}
Observe that \eqref{dbysub} yields
\begin{align}\label{intcal}
\int_{0}^{\infty}dr\,r^d\,p_1^{(\alpha)}(re_d)=(4\pi)^{-d/2}\int_{0}^{\infty}ds \,\eta^{(\alpha/2)}_1(s)s^{-d/2}\int_{0}^{\infty}dr\, r^{d}\exp\pthesis{{-\frac{r^2}{4s}}}.
\end{align}
Next, the change of variables $w=\frac{r^{2}}{4s}$ shows that
$$\int_{0}^{\infty}dr\, r^{d}\exp\pthesis{{-\frac{r^2}{4s}}}=4^{d/2}s^{\frac{d+1}{2}}\Gamma\pthesis{\frac{d+1}{2}}.$$
Thus, the desired result is obtained by replacing the last identity into the integral \eqref{intcal} and using \eqref{gen.expt.S1} with $\lambda=1/2$.
\end{proof}

The following theorem is the main result concerning the small time behavior of the heat content of $\dom$ in $\Rd$ when dealing with the heat kernels of  rotationally invariant $\alpha$-stable processes. It is remarkable that part $a)$ and $c)$ of the next result  are stronger than  Theorem 1.2 in \cite{Acu} since uniformly $C^{1,1}$-regular bounded domains have according to Example \ref{Smoothbd} finite perimeter.

\begin{thm}\label{mthm2} 
Consider $\dom \in \lm$ a bounded set with $\per<\infty$ and $d\geq 2$ integer.
\begin{enumerate}
\item[$ (a)$]Let $\alpha\in (1,2)$. Then, we have for all $t>0$ that 
\begin{align*}
\ld- \Halp\leq\frac{ t^{\frac{1}{\alpha}}}{\pi}\,\Gamma\pthesis{1-\frac{1}{\alpha}} \per.
\end{align*}
Furthermore,
\begin{align*}
\mylim{\tgo}{\frac{\ld- \Halp}{t^{\frac{1}{\alpha}}}}=\frac{1}{\pi}\Gamma\pthesis{1-\frac{1}{\alpha}}\per.
\end{align*}

\item[$(b)$] For $\alpha=1$, we have for all  $t<\lo=\sup\set{|x-y|:x,y\in \Omega}$ that
\begin{align*}
\ld- \mathbb{H}_{\Omega}^{(1)}(t)\leq t\pthesis{\lambda(\dom)+\frac{1}{\pi}\per
\ln\pthesis{\frac{1}{t}}}.
\end{align*}
Here,
$$\lambda(\dom)=\ld\, \lo^{-1}\,A_d\,\kappa_{d}+ \frac{\per}{\pi}\pthesis{\ln(\lo)+\int_{0}^{1}\frac{dr\,r^d}{(1+r^2)^\frac{d+1}{2}}},$$
and $\kappa_d$ and $A_d$ as given in \eqref{kappadef} and \eqref{vaunitball}, respectively.
In particular, we arrive at
\begin{align}\label{limsup}
\varlimsup_{\tgo}\frac{\ld- \mathbb{H}^{(1)}_{\Omega}(t)}{t\ln\pthesis{\frac{1}{t}}}\leq \frac{1}{\pi}\per.
\end{align}
In addition, the inequality \eqref{limsup} is sharp. That is, if B is the unit ball in $\Rd$, then
\begin{align}\label{sharp}
\mylim{\tgo}{\frac{\abs{B}- \mathbb{H}^{(1)}_{B}(t)}{t\ln\pthesis{\frac{1}{t}}}}=\frac{1}{\pi}Per(B).
\end{align}

\item[$(c)$] For $0<\alpha<1$, 
\begin{align*}
\mylim{\tgo}{\frac{\ld- \Halp}{t}}=
C_{\alpha,d}\,\,\myP_{\alpha}(\Omega),
\end{align*}
\end{enumerate}
where 
\begin{align}\label{alphaper}
\myP_{\alpha}(\Omega)=\int_{\dom}\int_{\dom^c}\frac{dx\,dy}{|x-y|^{d+\alpha}},
\end{align}
with $C_{\alpha,d}$ as given in \eqref{Adef}.
\end{thm}

\begin{rmk}
We point out that $\myP_{\alpha}(\Omega)$ defined in \eqref{alphaper} is called the $\alpha$-perimeter and it turns out to be linked with celebrated Hardy and isoperimetric inequalities.  We refer the  reader to the papers of Z. Q. Chen, R. Song \cite{Song} and R. L. Frank, R. Seiringer \cite{Frank} for  further results involving this quantity. In fact, it is shown in \cite{Frank} that there  exists $\lambda_{d,\alpha}>0$ such that
\begin{align*}
|\Omega|^{(d-\alpha)/d}\leq \lambda_{d,\alpha}\,\myP_{\alpha}\pthesis{\Omega},
\end{align*}
with equality if and only if $\Omega$ is a ball.
It is also proved in \cite{Lau} and \cite{FracP} that
\begin{align*}
\lim_{\alpha\downarrow 0}\alpha \myP_{\alpha}(\Omega)=d\,\,|B_1(0)|\,\,|\Omega|,\nonumber \\
\lim_{\alpha\uparrow 1}(1-\alpha)\myP_{\alpha}(\Omega)=K_{d}\,\,\mathcal{H}^{d-1}(\partial \Omega),
\end{align*}  
for some $K_d>0.$
\end{rmk}
{\bf Proof of Theorem \ref{mthm2}:} This theorem is basically a consequence of applying Theorem \ref{mthm} with $\beta=-\frac{d}{\alpha}$ and $\gamma=\frac{1}{\alpha}$. For the proof of part $a)$, we apply  part $(i)$ of Theorem \ref{mthm}   with $I_1=(1,2)$ due to Lemma \ref{finiteexp}. Regarding part $(c)$, we apply  part $(iii)$ of Theorem \ref{mthm}  with (based on \eqref{tcomp} and \eqref{tlim}) $\phi_{\alpha}(t)=t$, $\Lambda_{\alpha}(x)=c_{\alpha,d}\abs{x}^{-(d+\alpha)}$ and $J_{\alpha}(x)=C_{\alpha,d}\abs{x}^{-(d+\alpha)}$, where we need to make use of the fact that $\myP_{\alpha}(\Omega)<\infty$ if and  only if $\alpha\in (0,1)$ provided that $\per<\infty$, according to Corollary 2.13  in \cite{Lau}.

As far as part $(b)$, we appeal to part $(ii)$ of Theorem \ref{mthm} with $\kappa=\kappa_{d}$, $n=2$ and $m=\frac{d+1}{2}$. Notice that according to \eqref{vaunitball} and \eqref{kappadef}, we have
\begin{align}\label{picf}
\kappa_{d}\,w_{d-1}=\Gamma\pthesis{\frac{d+1}{2}}\pi^{-\frac{d+1}{2}}\cdot\frac{\pi^{\frac{d-1}{2}}}{\Gamma\pthesis{1+\frac{d-1}{2}}}=\frac{1}{\pi}.
\end{align}

 Concerning the limit 
\eqref{sharp}, we could apply Theorem 1.2 in \cite{Acu}, however to  prove that theorem, local coordinates around each point of the boundary are required which makes the proof very complicated. For this reason, in order to make this presentation as clear as possible, we provide a somewhat simple proof for \eqref{sharp} and we devote the next section to it. \qed

\section{heat content behavior for the poisson kernel \eqref{Cauchyk} over the unit ball.}\label{sec:poisson}

We want to study the small time asymptotic behavior of 
$$\myH_{B}^{(1)}(t)=\int_{\Rd}dx\, \ck{t}(x)g_{B}(x)=\int_{\Rd}dx\, \ck{1}(x)\,g_{B}(tx),$$
where $B\subset\Rd$ is the unit ball and $p_t^{(1)}(x)$ is the heat kernel described in \eqref{Cauchyk}. Observe that $B+z=B_1(z)=\set{x\in\Rd:|x-z|<1}$ so that $g_B(z)=|B\cap(B+z)|$
represents the volume of the intersection of two balls of radii one. It is also geometrically clear that $g_B(z)=g_B(Tz)$ for any orthonormal linear transformation $T$ on $\Rd$, which implies that $g_B$ is radial so that
\begin{align}\label{gradial}
g_B(z)=g_B(\abs{z}e_d),
\end{align}
with $e_d=(0,....0,1)\in \Rd$.

The following lemma provides a formula for the volume of  the intersection of two unit balls  in $\Rd$.
\begin{lem}\label{impl} 
Let $d\geq 2$ be an integer and  $0\leq a\leq 2$. 
Let $B=B_1(0)$ and $B(a)=B_1(ae_d)$. Then,  we have
\begin{align}\label{volball}
g_B(a\,e_d)=\abs{B\cap B(a)}=2\,A_{d-1}\,\Theta\pthesis{\sqrt{1-\frac{a^2}{4}}}-a\,w_{d-1}\,\pthesis{1-\frac{a^2}{4}}^{\frac{d-1}{2}},
\end{align}
where 
$$\Theta(z)=\int_{0}^{\arcsin(z)}d\theta\,\sin^{d-2}(\theta)\,\cos^{2}(\theta),$$
for $0\leq z\leq 1$ and $w_{d-1},A_{d-1}$ as defined in \eqref{vaunitball}. In particular, by taking $a=0$, we obtain
\begin{align}\label{theta1}
\Theta(1)=\frac{|B|}{2A_{d-1}}.
\end{align}
\end{lem}

\begin{proof} If $a=0$, the result is obvious. Assume $0<a\leq 2$.
We start by representing every point $x\in \Rd$ as $x=(\bar{x},x_d)\in \R^{d-1}\times \R$. It is not difficult to see that under these coordinates, we have
$$B\cap B(a)=\set{(\bar{x},x_d)\in \Rd: \abs{\bar{x}}\leq \sqrt{1-\frac{a^2}{4}},\,\, a-\sqrt{1-\abs{\bar{x}}^2}<x_d<\sqrt{1-\abs{\bar{x}}^2}}.$$
Therefore, by setting $\ell=\sqrt{1-\frac{a^2}{4}}$, we arrive at
\begin{align*}
\abs{B\cap B(a)}&=\int_{B_{\ell}(0)\subset \R^{d-1}}d\bar{x}\int_{a-\sqrt{1-\abs{\bar{x}}^2}}^{\sqrt{1-\abs{\bar{x}}^2}}
dx_d
=2\int_{B_{\ell}(0)\subset \R^{d-1}}d\bar{x}\,\sqrt{1-\abs{\bar{x}}^2}-a\abs{B_{\ell}(0)}\\
&=2A_{d-1}\int_{0}^{\ell}dr \,\,r^{d-2}\,\,\sqrt{1-r^2}-aw_{d-1}\ell^{d-1},
\end{align*}
where we have appealed to polar coordinates to obtain the last equality. Thus, the desired identity \eqref{volball} follows from the last  expression  and by performing the change of variable $r=\sin(\theta)$ in the integral term.
\end{proof}

Since $g_{B}(y)=0$ when $|y|\geq 2$ by Proposition \ref{gprop}  and $g_{B}(y)$ is radial (see \eqref{gradial}), we obtain by combining the identities $A_d=Per(B)$ and \eqref{picf} together with Lemma \ref{impl}    the following decomposition for $\myH_{B}^{(1)}(t)$.
\begin{align}\label{cydec}
\myH_{B}^{(1)}(t)&=A_d\int_{0}^{2t^{-1}}dr \,r^{d-1}\,\ck{1}(re_d)\,g_{B}(tr \,e_d)\\ \nonumber&=
N_1(t)-A_d\,\kappa_d\,w_{d-1}\,t\,N_2(t)=N_1(t)-\frac{1}{\pi}Per(B)\,t\,N_2(t),
\end{align}
where
\begin{align*}
\,\,\,\,\,\,\ \,\,\,\,\,\,\,\,\,\,\,\,\,\,\,N_1(t)&=2A_d\,A_{d-1}\kappa_d\int_{0}^{2t^{-1}}\frac{dr\, r^{d-1}}{(1+r^2)^{\frac{d+1}{2}}}\Theta\pthesis{\sqrt{1-\frac{t^2r^2}{4}}}
\end{align*}
and
\begin{align}\label{N2def}
\,\,\,\,\,\,\,\,\,\,\,\,\,\,\,\,\,\,\,\,N_2(t)&=\int_{0}^{2t^{-1}}\frac{dr\,r^d}{(1+r^2)^{\frac{d+1}{2}}}\pthesis{1-\frac{t^2r^2}{4}}^{\frac{d-1}{2}}.
\end{align}
Because $\Theta$ given in Lemma \ref{impl}  is an increasing function and $2A_{d-1}\Theta(1)=|B|$, we obtain that
\begin{align*}
N_1(t)\leq 2\,A_{d-1} \,\Theta(1)\,A_d\,\kappa_d\int_{0}^{\infty}\frac{dr\, r^{d-1}}{(1+r^2)^{\frac{d+1}{2}}}=|B|\cdot\lone{p_1^{(1)}}=\abs{B},
\end{align*}
so that by \eqref{cydec}, we have
\begin{align*}
\frac{1}{\pi}Per(B)\,t\, N_2(t)=N_1(t)-\myH_{B}^{(1)}(t)\leq \abs{B}-\myH_{B}^{(1)}(t).
\end{align*}
Observe that \eqref{limsup} tells us that
\begin{align*}
\varlimsup_{\tgo}\frac{\abs{B}- \mathbb{H}^{(1)}_{B}(t)}{t\ln\pthesis{\frac{1}{t}}}\leq \frac{1}{\pi}Per(B).
\end{align*}
Hence, to prove \eqref{sharp}, it suffices to show the following.
\begin{prop}\label{propint}
Consider the function $N_2(t)$  defined in \eqref{N2def}. Then,
$$1\leq \varliminf\limits_{\tgo}\frac{N_2(t)}{\ln\pthesis{\frac{1}{t}}}.$$
\end{prop}
\begin{proof}
We recall the following basic inequality. For $0<x<1$ and $d\geq 2$, we have
\begin{align}\label{bineq}
(1-x)^{\frac{d-1}{2}}\geq 1-\sigma_d x,
\end{align}
where $\sigma_d=\mathbbm{1}_{\set{2}}(d)+ \frac{(d-1)}{2}\cdot \mathbbm{1}_{[3,\infty)}(d)$. Based on the definition \eqref{N2def} of $N_2(t)$, we see that for any $t<2$ we have
\begin{align}\label{N2ineq}
N_2(t)&\geq \int_{1}^{2t^{-1}}\frac{dr\,r^d}{(1+r^2)^{\frac{d+1}{2}}}
\pthesis{1-\frac{t^2r^2}{4}}^{\frac{d-1}{2}}
=\int_{1}^{2t^{-1}}\frac{dr}{r}
\pthesis{1-\frac{t^2r^2}{4}}^{\frac{d-1}{2}} -\Phi(t)\\ \nonumber
&\geq  \int_{1}^{2t^{-1}}\frac{dr}{r}
\pthesis{1-\frac{\sigma_d}{8}\,t^2r^2}-\Phi(t)\geq  \ln\pthesis{\frac{1}{t}}-\frac{\sigma_d}{4}-\Phi(t),
\end{align}
where to obtain the second inequality, we have used $\eqref{bineq}$ and $\Phi(t)$ has been defined by
$$\Phi(t)=\int_{1}^{2t^{-1}}dr\pthesis{\frac{1}{r}-\frac{r^d}{(1+r^2)^{\frac{d+1}{2}}}}
\pthesis{1-\frac{t^2r^2}{4}}^{\frac{d-1}{2}}.$$
Next, notice that $\Phi(t)$ satisfies that
\begin{align}\label{phiineq}
\Phi(t)\leq \int_{1}^{\infty}dr
\pthesis{\frac{1}{r}-\frac{r^d}{(1+r^2)^{\frac{d+1}{2}}}}<\infty.
\end{align}
Finally, by combining  \eqref{N2ineq} with \eqref{phiineq}, we arrive at the desired result and this finishes the proof of  Proposition \ref{propint}.
\end{proof}

\bigskip
{\bf Acknowledgements}: I would like to thank the referee whose comments and corrections have improved the quality  of the paper. This  investigation has been supported by Universidad de Costa Rica, project 1563.
\\

\end{document}